\documentclass[11pt]{amsart} %

\usepackage{amsmath, amsthm}
\usepackage{amssymb, mathtools}
\usepackage{extarrows}

\usepackage[dvipsnames]{xcolor}

\usepackage{bigfoot}

\usepackage[shortlabels]{enumitem}

\usepackage{todonotes}

\usepackage{hyperref}

\usepackage{tikz-cd}

\usepackage{endnotes}

\usepackage{booktabs}
\usepackage{tabulary}

\usepackage{soul}
\usepackage{scalerel}

\usepackage[utf8]{inputenc}
\usepackage[greek,english]{babel}
\usepackage{lmodern}

\newtheorem{THEOREM}{Theorem}[section]

\newtheorem{Conclusion}[THEOREM]{Conclusion}

\newtheorem{Theorem}[THEOREM]{Theorem}
\newenvironment{theorem}{\begin{Theorem}}{\end{Theorem}}

\newtheorem{Lemma}[THEOREM]{Lemma}
\newenvironment{lemma}{\begin{Lemma}}{\end{Lemma}}

{\theoremstyle{definition}
\newtheorem{notation}[THEOREM]{Notation}
\newtheorem{definition}[THEOREM]{Definition}

\newtheorem{remark}[THEOREM]{Remark}

}

\newtheorem{Claim}[THEOREM]{Claim}

\newtheorem{Subclaim}[THEOREM]{Subclaim}

\newtheorem{Corollary}[THEOREM]{Corollary}
\newenvironment{corollary}{\begin{Corollary}}{\end{Corollary}}

\newtheorem{Proposition}[THEOREM]{Proposition}
\newenvironment{proposition}{\begin{Proposition}}{\end{Proposition}}

\def\>{\rangle}
\def\<{\langle}

\newenvironment{eqn}{\begin{equation*}}{\end{equation*}}

\newcommand{\cf}{\mathop{\rm cf}}   
\newcommand{\rge}{\mathop{\rm rge}} 
\newcommand{\dom}{\mathop{\rm dom}} 
\newcommand{\AC}{\mathop{\rm AC}} 
\newcommand{\DC}{\mathop{\rm DC}} 
\newcommand{\Ha}{\mathop{\rm H\hskip0.05pt}} 
\newcommand{\CUT}{\mathop{\rm CUT}} 
\newcommand{\FA}{\mathop{\rm FA}} 
\newcommand{\Coll}{\mathop{\rm Coll}} 
\newcommand{\fin}{\mathop{\hbox{\scshape Fin}}} 
\renewcommand{\lim}{\mathop{\rm lim}} 
 
\newcommand{\otp}{\mathop{\rm otp}} 

\newcommand{\On}{\mathop{\rm On}}
\newcommand{\lh}{\mathop{\rm lh}}
\newcommand{\height}{\mathop{\rm ht}}
\newcommand{\ssup}{\mathop{\rm ssup}} 
\newcommand{\concat}{\kern-.25pt\raise4pt\hbox{$\frown$}\kern-.25pt}
\newcommand{\restrict}{\upharpoonright}  
\newcommand{\on}{\upharpoonright}  
\newcommand{\Pbb}{\mathbb{P}}
\newcommand{\Qbb}{\mathbb{Q}}
\newcommand{\Tbb}{\mathbb{T}}
\newcommand{\proves}{\vdash}
\def\sss{\hskip2pt}  
\def\ssss{\hskip1pt}
\def\tupof#1#2{\<\ssss #1\sss|\sss #2\ssss\>} 
\def\tup#1{\<\ssss #1\ssss\>} 
\def\setof#1#2{\{\ssss #1\sss|\sss #2\ssss\}} 
\def\set#1{\{\ssss #1\ssss\}} 

\newcommand{\cardrule}{\hrule height.2pt}
\newcommand{\cardrulefill}{\cleaders\cardrule\hfill}
\newcommand{\cardchar}[1]{\vbox{\ialign{##\crcr
    \cardrulefill\crcr\noalign{\kern1pt\nointerlineskip}
    $\hfil\displaystyle{#1}\hfil$\crcr}}}
\newcommand{\card}[1]{\cardchar{\cardchar{#1}}}
\newcommand{\scardchar}[1]{\vbox{\ialign{##\crcr
    \cardrulefill\crcr\noalign{\kern1pt\nointerlineskip}
    $\hfil\scriptstyle{#1}\hfil$\crcr}}}
\newcommand{\scard}[1]{\scardchar{\scardchar{#1}}}

\newcommand\extrafootertext[1]{%
    \bgroup
    \renewcommand\thefootnote{\fnsymbol{footnote}}%
    \renewcommand\thempfootnote{\fnsymbol{mpfootnote}}%
    \footnotetext[0]{#1}%
    \egroup
}

\newcommand{\Addresses}{{
  \bigskip
  \footnotesize

  \textsc{Departmento de Matem\'atica, Universidade Federal de Bahia,
    Av. Adhemar de Barros, S/N, Ondina, Salvador, Bahia, Brazil. CEP: 40170-110.}\par\nopagebreak

  \textit{E-mail address}, Diego Lima Bomfim: \texttt{bomfim.diego@ufba.br@ufba.br}
  \vskip-9pt
  \textit{E-mail address}, Charles Morgan: \texttt{charlesmorgan@ufba.br}
  \vskip-12pt
  \textit{E-mail address}, Samuel Gomes da Silva: \texttt{samuel@ufba.br}
}}

\let\bemph\relax 
\DeclareTextFontCommand{\bemph}{\bfseries\em}

\title{On forcing axioms and weakenings of the Axiom of Choice}
\author{Diego Lima Bomfim, Charles Morgan and Samuel Gomes da Silva}

\begin{document}

\begin{abstract}
  We prove forcing axiom equivalents of two families of weakenings of
  the axiom of choice: a trichotomy principle for cardinals isolated
  by L\'evy, $\Ha_\kappa$, and $\DC_\kappa$, the principle of
  dependent choices generalized to cardinals $\kappa$, for regular
  cardinals $\kappa$.  Using these equivalents we obtain new forcing
  axiom formulations of $\AC$.

A point of interest is that we use a new template for forcing axioms. For the class of forcings
to which we asks that the axioms apply, we do not ask that they apply to all collections of
dense sets of a certain cardinality, but rather only for each particular forcing to a specific
family of dense sets of the cardinality in question.
\end{abstract}
\maketitle

\extrafootertext{2020 Mathematics Subject Classification: {03E25, 03E57, 03E05}.}
\extrafootertext{Keywords: dependent choice, Levy's trichotomy, forcing axioms.} 
\extrafootertext{Date: \date{\today}}
\extrafootertext{This research was carried out whilst the second author was a Professor Visitante at the
  Universidade Federal de Bahia (UFBA), under EDITAL PV 001/2019 - PROPG/ UFBA.
  His grateful for the support of UFBA under this edital and for the kind hospitality of the Departamento de Matemática.}

\vfill\eject

\section{Introduction}\label{intro}
This paper had its origins in parts of the first author's master's
thesis, \cite{Bomfim22}, written under the supervision of the third
author. These sections of \cite{Bomfim22} were inspired by
\cite{Parente15} and \cite{Viale19} which, as Viale writes in the
latter article, ``focus on the interplay between forcing axioms and
various other non-constructive principles widely used in many fields
of abstract mathematics, such as the axiom of choice and Baire's
category theorem.''

One of the original objectives of the work recounted in
\cite{Bomfim22} was to find a forcing axiom equivalent of the
statement that every infinite set is Dedekind-infinite.  Subsequent
discussions between the three authors led to the broader consideration
of the interplay between forcing axioms and weakenings of the Axiom of
Choice which forms the focus of the paper.

One early paper in which various families of weakenings of the axiom of choice were considered is
L\'evy's \cite{Levy64}.  L\'evy focused, in particular, on what are nowadays called
$\AC_\alpha$ and $DC_\alpha$,\footnote{See below for definitions.} for ordinals $\alpha$,
and on a family of statements generalizing the axiom that every infinite set is Dedekind-infinite.

We remind the reader of the standard notation surrounding injective
embeddings of one set into another. Doing so allows us to introduce
this generalization immediately.
\vskip12pt

\begin{notation}  Suppose $X$ and $Y$ are sets. We write $X\preccurlyeq Y$ if there is an
  injection of $X$ into $Y$.
  We write $X\prec Y$ if $X\preccurlyeq Y$ and $Y\not\preccurlyeq X$. We write
  $X\approx Y$ if $X\preccurlyeq Y$ and $Y\preccurlyeq X$.
\end{notation}

\begin{definition} Let $X$, $Y$ be sets. We say 
  \emph{trichotomy} (for cardinals) holds for $X$ and $Y$ if
  $X\prec Y$ or $X\approx Y$ or $Y\prec X$, or, equivalently,
  $X\preccurlyeq Y$ or $Y\preccurlyeq X$,
  or $X\prec Y$ or $Y\preccurlyeq X$.
  \end{definition}

The statement `every infinite set is Dedkind-infinite' can be rephrased as
`every set is either finite or there is an injection of a $\omega$ into it'.
That is, for every set $X$ trichotomy holds between $X$ and $\omega$ (in the form
$X\prec \omega$ or $\omega\preccurlyeq X$).
As noted above, L\'evy considered a generalization of this statement.

\begin{definition} (\emph{cf.}~\cite{Levy64})
Let $\alpha$ be an ordinal.  \hbox{$\Ha_{\alpha}$  is the principle: \hskip12pt}
\[\forall X \sss\sss (X\prec \alpha\hbox{ or }\alpha\preccurlyeq X) \]
that is, for every set $X$ trichotomy holds between $X$ and $\alpha$:

\end{definition}

\begin{remark} \cite{Herrlich06} writes $\fin$ for $\Ha_\omega$, a notation used in \cite{Bomfim22}.
\end{remark}
\vskip6pt

As we recall below in Proposition (\ref{Hartogs_theorem}),
Hartogs (\cite{Hartogs1915}) proved $\AC \Longleftrightarrow$
for all ordinals $\alpha$ we have $\Ha_\alpha$. We speculate that
Lévy used the letter `H' as a glancing reference to Hartogs's work.

Our main results are to prove forcing axiom equivalents of $\Ha_\kappa$ (and consequently of
$\Ha_\omega$ -- $\fin$) and $\DC_\kappa$ for regular cardinals $\kappa$. Using these equivalents
we obtain new forcing axiom formulations of $\AC$.

A point of interest is that we use a new template for forcing axioms. For the class of forcings
to which we asks that the axioms apply, we do not ask that they apply to all collections of
dense sets of a certain cardinality, but rather only for each particular forcing to a specific
family of dense sets of the cardinality in question.

\vskip12pt

We now introduce some further notation, definitions and results, for the most part
standard for set theory, which will be used in the course of the paper.

As the paper concerns weakenings of the axiom of choice our base theory is ZF and all results, unless otherwise stated,
are ZF-results.

\begin{definition} A cardinal which is an ordinal is an \emph{initial ordinal} or an \emph{aleph}.
  \end{definition}

\begin{notation} Let $X$ be a set. We write $\card{X}$ for the cardinality of $X$.
If $X$ is a set we write $h(X)$ for its Hartogs number, the cardinal which is the least ordinal
  $\alpha$ for which there is no injection of $\alpha$ into $X$.
\end{notation}

\begin{remark} In the ensuing we will be particularly interested in the statement $X\not\prec \kappa$ for
  initial ordinals $\kappa$. This says that there is no injective embedding of $X$ into any well-ordered
  set of cardinality less than $\kappa$.
  \end{remark} 

The following lemma is trivial, however we make heavy use of the
negation of the right hand side of the conclusion as an equivalent of
the negation of the left hand side.

\begin{lemma}\label{trivial_lemma} Let $X$ be a set and $\kappa$ an initial ordinal.
  Then $X\prec \kappa$ if and only if there is some $\beta<\kappa$ and
  an injection from $\beta$ into $X$ which is surjective.
\end{lemma}

\begin{proof} If $X\prec \kappa$ it is well-orderable and $\card{X}<\kappa$, so there is a bijection from
  some cardinal less than $\kappa$ to $X$. If there is a bijection
  from some cardinal less than $\kappa$ to $X$ then the inverse
  demonstrates that $X\prec \kappa$.
  \end{proof}

\begin{notation} Let $X$ be a set. Then $[X]^{\kappa}$ is
$\setof{Y\subseteq X}{\card{Y}=\kappa}$ and $[X]^{<\kappa}$ is
  $\setof{Y\subseteq X}{\card{Y}<\kappa}$. If $\alpha\in\On$ we write
    ${^{\alpha}X}$ for $\setof{f}{f:\alpha\longrightarrow X}$ and
    ${^{<\alpha}X}$ for $\setof{f}{\exists \beta< \alpha \sss\sss f:\beta\longrightarrow X}$.
\end{notation}

We introduce two \emph{non-standard} pieces of notation.

\begin{notation}
    We write $(X)^{\alpha}$ for
    $\setof{f}{f:\alpha\longrightarrow X\hbox{ and }f\hbox{ is an injection}}$ and
    $(X)^{<\alpha}$ for $\setof{f}{\exists \beta < \alpha \sss\sss f:\beta\longrightarrow X\hbox{ such that
      }f\hbox{ is an injection}}$.
\end{notation}

    (Note that typically in set theory $(X)^{\alpha}$ denotes
    the collection of \emph{increasing} sequences of length $\alpha$
    of elements of a set $X$ which is equipped with a strict partial
    ordering -- equivalently the collection of strictly increasing injections from $\alpha$ to $X$.
    Here we do not assume any such ordering exists.)

\begin{notation}
    For a sequence of sets $t$ we write $\lh(t)$ for $\dom(t)$ and
    if we are thinking of $t$ as an element of a tree of sequences we, further, write $\height(t)$ for $\lh(t)$.
\end{notation}

\begin{definition} If $Y$ is a set of ordinals, $\ssup(Y)$ (the \emph{strong supremum} of $Y$)
  is the least ordinal greater than all the elements of $Y$: if $Y$ has a maximal element then
  $\ssup(Y) = \bigcup Y +1$, and if not then $\ssup(Y)=\bigcup Y$.
\end{definition}

Later in the paper we will need the notion of (possibly infinite) concatenation.

\begin{definition}  Let $\tau$ be an ordinal (possibly infinite)
  and let $\langle \tau_i : i < \tau \rangle$ be a sequence of non-zero ordinals.
  For convenience, define $\tau^* = \tau+1$ if $\tau$ is a sucessor ordinal, and $=\tau$ if $\tau$ is a limit.
  
  From $\langle \tau_i : i < \tau \rangle$ we define 
  a sequence of ordinals $\langle \sigma_i : i < \tau^* \rangle$ by induction on $\tau^*$.  Let $\sigma_0 = 0$.
  Given $\sigma_i$, let $\sigma_{i+1} = \sigma_i + \tau_{i}$.  If $i$ is a limit ordinal let $\sigma_i = \bigcup_{j<i} \sigma_j$.

  Now let $X$ be a set and suppose, for each $i<\tau$ and $j<\tau_i$, that $t_{ij}$ is an element of $X$.
  Let $t_i = \langle t_{ij} : j < \tau_i \rangle$ and let $t = \langle t_i : i < \tau \rangle$. 

Now we define a sequence $t^c$ of length $\ssup(\{\sigma_i : i <
\tau^* \})$ by $t^c_{\sigma_i + j} = t_{ij}$ for every $i<\tau$ and
$j<\tau_i$. We call $t^c$ \emph{the concatenation of} $t$.
\end{definition}

Finally we state $\AC_\alpha$ for the reader's convenience. We define and discuss
$\DC_\alpha$ in the following section, immediately below.

\begin{definition} ($\AC_\alpha$) Let $\alpha$ be an ordinal. $\AC_\alpha$ is the axiom that
  whenever ${\mathcal X} = \setof{X_\beta}{\beta<\alpha}$ is
  a collection of sets such that for all $\beta<\alpha$ we have $X_\beta\ne\emptyset$, there
  is a function $f:\alpha\longrightarrow \bigcup {\mathcal X}$ such that for all
  $\beta<\alpha$ we have $f(\beta) \in X_\beta$.
  \end{definition}

\vskip12pt

\section{Variants of the Axiom of Dependent Choices}\label{DC_varts}

We recall the Axiom of Dependent Choice ($DC$) and the axiom $\DC_{\alpha}$.

\begin{definition} 
  $\DC$ is the \emph{axiom of dependent choice}:\hfill\break\vskip-18pt
  for every set $X$ and
    every binary relation $R$ on $X$ for which for all $x\in X$ there is
    some $y\in X$ with $x\sss R\sss y$ there is a function
    $f:\omega\longrightarrow X$ such that for all $i<\omega$ we have
    $f(i) \sss R \sss f(i+1)$.
\end{definition}

\begin{definition}
    If $\alpha$ is an ordinal then $\DC_{\alpha}$ is the axiom:\hfill\break\vskip-18pt
    if $F:{^{<\alpha}X}\longrightarrow {\mathcal P}(X) \setminus
    \set{\emptyset}$ there is a function $g:\alpha\longrightarrow X$ such
    that for all $i<\alpha$ we have $g(i) \in F(g\on i)$.
  \end{definition}

It is well known that $\DC$ and $\DC_\omega$ are equivalent (under ZF).
However, the two equivalents inspire separate ways to weaken the axiom's force.

\vskip12pt

For $\DC$ itself we can consider restrictions on the class of binary relations.
In order to give examples of these we need a little more notation.

\begin{definition} Let $P=(\Pbb,<_\Pbb)$ be a pre-order. $P$ has the
  \emph{finite predecessor property (fpp)} if each element of $\Pbb$
  has finitely many predecessors under $<_\Pbb$. That is, for
  $p\in\Pbb$ we have $\setof{q\in\Pbb}{p <_\Pbb q}$ is finite. An
  element $p$ of $\Pbb$ is \emph{maximal} if it has no predecessors
  and \emph{minimal} if there is no $q\in \Pbb$ such that $q<_\Pbb p$.
\end{definition}

\begin{definition} Let $P=(\Pbb,<_\Pbb)$ be a partial order. $P$  is \emph{strict} if for all $x$ we have
  $x \not <_\Pbb x$. $P$ is \emph{reticulated} or \emph{a lattice} if every pair of elements has a sup and an inf:
for all $p$, $q\in \Pbb$, there are $p \wedge q$, $p\vee q$ where if $r < p$, $q$ then
$r < p \wedge q$ or $r = p \wedge q$ and if
$p$, $q < r$ then $p \vee q < r$ or $p \vee q = r$.
  \end{definition}

Diener (\cite{Diener94}) showed that $\DC$ is equivalent to
the formal restriction of $\DC$ to the class of relations that are (strict) partial orders.
Gomes da Silva \cite{Silva12} considered the following (true) weakening of $\DC$.

\begin{notation}
  \begin{itemize}
  \item[] $\DC_{l,fp}$ is the axiom of dependent choice
    restricted to strict lattices with the fpp.
    More explicitly, $\DC_{l,fp}$ is the principle that for every set
    $X$ and strict lattice relation $R$ on $X$ with the fpp,
    if additionally for all $x\in X$ there is some $y\in X$ such
    that $x \sss R \sss y$ there is a function
    $f:\omega\longrightarrow X$ such that for all $i<\omega$ we have
    $f(i) \sss R \sss f(i+1)$.
  \end{itemize}
  \end{notation}
\vskip12pt

However, the formulation of $\DC_\alpha$ also naturally leads to weakenings given by
restricting the class of functions $F$ for which the axiom applies.
One example example of such restriction is the collection of
axioms $\Ha_\alpha$ introduced in \S{}\ref{intro}.

In order to see why we think of $\Ha_{\alpha}$ as a variant of
$\DC_{\alpha}$ note it can be rewritten as the statement that for
every $X$ such that $X \not\prec \alpha$ there is a function
$g:\alpha\longrightarrow X$ such that for all $i<\omega$ we have $g(i)
\in X\setminus \rge(g\on i)$.  We can think of this formulation as
$\DC_{\alpha}$ restricted to the special cases (one for each set $X$)
of the functions $F^X_{seq}:{^{<\alpha}X}\longrightarrow {\mathcal
  P}(X) \setminus \set{\emptyset}$ given by $F^X_{seq}(t) = X\setminus
\rge(t)$. The requirement that $X \not\prec \alpha$ is (using Lemma
\ref{trivial_lemma}) exactly what is necessary to show that
$F^X_{seq}$ is well defined. For $F^X_{seq}$ is only not defined
everywhere if there is some $t\in {^{<\alpha}X}$ with $F^X_{seq}(t) =
\emptyset$, which in turn says that $t$ is an injection of some
ordinal less than $\alpha$ into $X$ which is surjective.

Perhaps more obviously, for ordinals $\alpha$ one can also see
$\Ha_\alpha$ as an instance of trichotomy for between pairs consisting of an ordinal
and an arbitrary set. 
\vskip12pt

\begin{proposition}\label{Hartogs_theorem}\cite{Hartogs1915} The following are equivalent
  \begin{enumerate}[align=parleft,labelsep=1cm, label=(\roman*)]
  \item  $\AC$,
  \item For all ordinals $\alpha$, $ \Ha_\alpha$ holds,
    \item There is a proper class of ordinals  $\alpha$ for which $\Ha_\alpha$ holds.
  \end{enumerate}
 \end{proposition}

\begin{proof} $\AC$ is equivalent to $\hbox{WO}$,
  the principle that every set can be
  well ordered. This immediately gives us that $\Ha_{\alpha}$ holds
  for every ordinal $\alpha$ by the trichotomy theorem for
  ordinals.

  If $H_\kappa$ holds for all initial ordinals it holds for a proper class of initial ordinals.

  We show the third item implies every set is well-orderable.
  Given any set $X$ let $\alpha \ge h(X)$, the Hartog's number for $X$, be such that $\Ha_\alpha$ holds.
  Applying  $\Ha_{\kappa}$ to 
  hence a well-ordering of $X$.
  \end{proof} 


\vskip12pt

We first of all observe that, as shown by Lévy (\cite{Levy64}),
we may restrict our attention to the cases where $\alpha$ is a cardinal.
We give proofs as the proofs in \cite{Levy64} are indirect.

\begin{proposition} \cite{Levy64}
 For any ordinal $\alpha$ we have $\Ha_\alpha$ if and only if $\Ha_{\scard{\alpha}}$.
\end{proposition}

\begin{proof} For any ordinal $\alpha$ we have a bijection between $\alpha$ and $\card{\alpha}$. Hence
  if we have for every $X$ that $X\prec \alpha$ or $\alpha\preccurlyeq X$
  then for every $X$ we have $X\prec \alpha\preccurlyeq \card{\alpha}$ or $\card{\alpha} \preccurlyeq \alpha\preccurlyeq X$.
  Similarly if for all $X$ we have $X\prec \card{\alpha}$ or $\card{\alpha}\preccurlyeq X$ then
  for all $X$ we have $X\prec \card{\alpha}\preccurlyeq {\alpha}$ or ${\alpha} \preccurlyeq \card{\alpha}\preccurlyeq X$
  \end{proof}

\begin{proposition}\label{Levy_sing}\cite{Levy64} Let $\alpha$ be a singular limit ordinal.
  Then \[ ZF\proves (\forall \beta<\alpha \sss\sss DC_\beta) \Longrightarrow DC_\alpha.\]
\end{proposition}

\begin{proof} Let $\langle \alpha_\xi : \xi < \cf(\alpha) \rangle$
  be an increasing, continuous sequence of ordinals with $\alpha_0=0$ and
$\bigcup_{\xi<\cf(\alpha)} \alpha_\xi = \alpha$. For
  $\xi<\cf(\alpha)$ let $\gamma_\xi =  \otp([\alpha_\xi,\alpha_{\xi+1}))$.

Let $F:{}^{<\alpha}X \longrightarrow {\mathcal P}(X) \setminus \{ \emptyset \}$.
Define {${G: {}^{<\cf(\alpha)}({}^{<\alpha}X) \longrightarrow {\mathcal P}({}^{<\alpha}X) \setminus \{ \emptyset \}}$} by
for $t$ such that $\lh(t)=\xi$, and recalling the notation that $t^c$ is the
concatenation of $t$, setting 
\begin{eqn}
\begin{split}
  G(t) = & \{ s \in {^{\gamma_\xi}X}  : \forall j < \gamma_\xi \; s_j \in F(t^c \concat s\restrict j) \}
  \hbox{ if }\forall i < \xi \sss\sss \otp(t_i)=\gamma_i\hbox{, and}\\
       = & \{ \langle \rangle \} \hbox{ (the set whose only element is the empty sequence) otherwise.}
\end{split}
\end{eqn}

Note: the definition of $G$ uses $DC_\beta$ for $\beta\le
\alpha_{\xi+1}$ to ensure that the $G(t)$ are non-empty in the
substantive cases.

Let $h:\cf(\alpha) \longrightarrow {}^{<\alpha}X$ be given by $DC_{\cf(\alpha)}$ applied to $G$ and let $g = h^c$.

Then $g:\alpha\longrightarrow X$ is a function and $\forall \beta<\alpha \; g(\beta) \in F(g\restrict \beta)$.
\end{proof}
\vskip12pt

\begin{corollary} \cite{Levy64} If $\alpha$ is an ordinal then
$\DC_{\alpha}$ holds if and only if $\DC_{\scard{\alpha}}$ does. If $\alpha$ is a singular initial ordinal
then $\DC_{\alpha}$ holds if and only if for all regular $\beta<\alpha$ we have that $\DC_{\beta}$ holds.
\end{corollary}

\begin{proof} For the first statement use induction on $[\card{\alpha},\card{\alpha}^+)$, applying
    Proposition (\ref{Levy_sing}) at limit ordinals. The second statement follows from Proposition (\ref{Levy_sing}) and the first.
  \end{proof}
\vskip6pt

\begin{remark} As remarked by L\'evy in \cite{Levy64}, if $\alpha$ is a limit cardinal,
  $\Ha_\beta$ holds for all $\beta<\alpha$ 
  and $\AC_{\cf(\alpha)}$ holds then $\Ha_\alpha$ holds.

However, L\'evy (\cite{Levy64}, for ZFA) and Pincus (\cite{Pincus69}, for ZF) gave models in which 
$\Ha_\beta$ holds for all $\beta < \alpha$ and $DC_\gamma$ holds for all $\gamma<\cf(\alpha)$,
yet $\Ha_\alpha$ (and hence also $\AC_{\cf(\alpha)}$) fails. See, also, \cite{Jech73}, Theorem (8.6) for a presentation
of these results.
\end{remark}

As a consequence of these results,
we restrict attention here to $\DC_{\kappa}$ and $\Ha_\kappa$ for cardinals $\kappa$.

We now take the opportunity to introduce some notation relevant to forcing axioms.

\begin{definition} Let $\Pbb$ be a partial order, and
  ${\mathcal D}$ a collection of subsets of $\Pbb$.
  A filter $G$ on $\Pbb$ is \emph{$\Pbb$-generic for ${\mathcal D}$} if
  for every $D\in {\mathcal D}$ we have $G\cap D\ne \emptyset$.
  We say $\FA(\Pbb,{\mathcal D})$ holds if there is
  $\Pbb$-generic filter for ${\mathcal D}$.
  We say $\FA_{\mathcal D}(\Pbb)$ holds if  ${\mathcal D}$ is a collection of dense subsets of $\Pbb$ and
  $\FA(\Pbb,{\mathcal D})$ holds.
  \end{definition}

\begin{definition} Let ${\mathcal P}$ be a class of partial orders and $\kappa$ a cardinal.
  We say $\FA_\kappa({\mathcal P})$ holds if for every $\Pbb\in {\mathcal P}$ and
  collection ${\mathcal D}$ of $\kappa$-many dense subsets of $\Pbb$ we have that
  $\FA(\Pbb,{\mathcal D})$ holds.
  \end{definition}

An intermediate notion, essentially a uniform version of the first
definition across a class of partial orders, is the following.

\begin{definition} Let ${\mathcal P}$ be a class of partial orders and
  let ${\mathfrak D}$ be a uniformly definable Function for which
  there is some cardinal $\kappa$ such that if $\Pbb \in {\mathcal P}$ then
  ${\mathfrak D}(\Pbb) = \tupof{{\mathfrak D}(\Pbb)_i}{i<\kappa}$ is a collection of
  subsets of $\Pbb$. We say $\FA({\mathcal P},{\mathfrak D})$ holds if
  for every $\Pbb\in {\mathcal P}$ we have that
  $\FA(\Pbb,{\mathfrak D}(\Pbb))$ holds.  We say $\FA_{\mathfrak D}({\mathcal P})$ holds if 
  for every $\Pbb\in {\mathcal P}$ we have that $\FA_{{\mathfrak D}(\Pbb)}(\Pbb)$ holds.
\end{definition}

    \section{Results from \cite{Bomfim22}}.

\begin{definition} $\fin$ is the statement that every infinite set is Dedekind-infinite.
  That is, whenever $X$ is infinite there is an injection
  $f:\omega\longrightarrow X$, or equivalently that $X$ has a
  countable, infinite subset.
  \end{definition}

   As remarked in the introduction, one of the original motivations of
   the work recounted in \cite{Bomfim22} was to find a forcing axiom equivalent of $\fin$. 
   We give here a sketch of the principal results of \cite{Bomfim22} on this problem.

We introduce some relevant weakenings of the axiom of choice.
  \begin{itemize}
  \item[] $\AC_{\omega}(\fin)$ is the
    axiom of choice for countably families of finite sets.
  \item[] $\CUT(\fin)$ is: every countable union of finite sets is countable.
  \end{itemize}

  We also introduce notation for a specific class of forcings.

  \begin{notation} Let $\Gamma$ be the class of pre-orders which are countable unions of
    finite pre-orders. \end{notation}

  Previously, Shannon (\cite{Shannon90}),
  Herrlich (\cite{Herrlich06} and Gomes da Silva (\cite{Silva12}) had showed the following. 

  \vskip12pt

  \begin{proposition} (\cite{Shannon90}, Corollary 2)
    \hskip10pt $\AC_{\omega}(\fin) \Longleftrightarrow \FA_\omega(\Gamma)$.
  \end{proposition}
  \vskip12pt
  
  \begin{proposition} (\cite{Herrlich06})
    \hskip10pt $ \AC_{\omega}(\fin) \Longleftrightarrow \CUT(\fin)$.
  \end{proposition}

  \begin{proposition} (\cite{Silva12})
    \hskip10pt $\fin \Longleftrightarrow \DC_{l,fp} + \AC_{\omega}(\fin)$.
  \end{proposition}

  \cite{Bomfim22} revisited the forcing axiom $\FA_\omega(\Gamma)$
  used in \cite{Shannon90} and showed $\fin$ is a 
  consequence of the combination of this forcing axiom and another
  forcing axiom. In order to state this result it is convenient to
  introduce notation for a another class of forcings.

  \begin{notation} Let $\Lambda$ be the class of trees $T$ such that there is a
    strict lattice ${\mathcal L}$ with fpp and no minimal elements and the elements of $T$
    are finite, strictly decreasing sequences from ${\mathcal L}$, ordered by extension.
  \end{notation}
  \vskip6pt 

  \begin{theorem}\label{Diego325} (\cite{Bomfim22}, Theorem 3.2.5) \hskip10pt
    $\FA_\omega(\Lambda) + \FA_\omega(\Gamma) \Longrightarrow \fin$.
  \end{theorem}


\section{Principal results}
  
  \begin{notation}    Let $\kappa$ be a regular cardinal. 

    For each set $X$ with $X \not\prec \kappa$ let
    $\Coll(\kappa,X)$ be
    the set $(X)^{<\kappa}$, that is, the set $\setof{f\in {{^{<\kappa}}X}}{f\hbox{ is injective}}$, partially ordered by extension:
    \[ {g\le f \Longleftrightarrow \dom(f) \le \dom(g) \sss\sss\&\sss\sss g\on\dom(f) = f}.\]
    For $i<\kappa$ set $L^X_i = \setof{f\in \Coll(\kappa,X)}{i\le \dom(f)}$. 

    Let ${\mathit Coll}_\kappa$ be the class of partial orders of the form
    $\Coll(\kappa,X)$ for $X$ with $X \not\prec \kappa$.
    Let ${\mathfrak L}^\kappa$ be the Function such that for each set $X$ with $X \not\prec \kappa$
    we have ${\mathfrak L}^\kappa(\Coll(\kappa,X)) = \tupof{L^X_i}{i<\kappa}$. 
 
  \end{notation}
  \vskip12pt

 \begin{remark}\label{rmk_on_defn_coll} The elements of $\Coll(\kappa,X)$ as defined here are injective functions. In many presentations of forcing under ZFC
    the forcing to collapse (the cardinality of) a set is taken to be the set of all partial functions from the cardinal to which the cardinality is to be collapsed
    to the target set. See, for example, \cite{Kunen80}, Definition (6.1). This makes no difference in the context of forcing extensions under ZFC,
    but does play a r\^ole for us here -- see the Remark (\ref{rmk_more_on_defn_coll}) following Proposition (\ref{FA_coll_iff_DC_seq}).
\end{remark}
  
  \begin{lemma}\label{injection_into_X_implies_density} Let $\kappa$ be a  regular cardinal and  $X$ a set with $X \not\prec \kappa$.
   For each $i<\kappa$ we have $L^X_i $ is
   an open subset of $\Coll(\kappa,X)$.
    If there is an injection  $g:\kappa\longrightarrow X$ then each
    $L^X_i$ is dense in $\Coll(\kappa,X)$.
  \end{lemma}

  \begin{proof} It is clear that $L^X_i $ is
   an open subset of $\Coll(\kappa,X)$. Given an injection  $g:\kappa\longrightarrow X$, $i<\kappa$ and
   $p\in \Coll(\kappa,X)$ with $\dom(p)=\alpha < i $ there is some $\beta<\kappa$ such that 
  $\rge(p)\cap \rge(g) \subseteq \rge(g\on \beta)$. Consequently if $q =p\concat g\on [\beta,\beta+i)$ then $q \le p$ and $q\in L^X_i$. 
\end{proof}
\vskip12pt

   Now we can state our first observations.

   \begin{proposition} $\FA({\mathit Coll_\omega},{\mathfrak L^\omega}) \Longleftrightarrow \fin$.
    \end{proposition}

   \begin{proof} ``$\Longrightarrow$''. Let $X$ be an infinite set. First of all, we argue that
     for each $n<\omega$ we have that $L^X_n\cap \Coll(\omega,X)$ is dense in $\Coll(\omega,X)$. Given
     $p\in \Coll(\omega,X)$ with $\dom(p) = m$, either $n\le m$, in which case
     $p\in L^X_n$, or $m < n$ and 
     we can choose a set $x \in [X\setminus \rge(p)]^{n-m}$ and extend $p$ to a condition $q$ with $\dom(q)=n$, and hence  $q\in L^X_n$,
     by  letting $q\on n\setminus m$ enumerate $x$.

     Furthermore, if $G$ is $\Coll(\omega,X)$-generic for $\tupof{L^X_n}{n<\omega}$ as given by invoking $\FA(\Coll(\omega,X),\tupof{L^X_n}{n<\omega})$ then
    $g=\bigcup G:\omega\longrightarrow X$ is injective. 

    ``$\Longleftarrow$''. Suppose $X$ is an infinite set. By $\fin$,
    let $g:\omega\longrightarrow X$ be an injection.
    Set $G = \setof{g\on n}{n\in \omega}$. Then $G$ is
    $\Coll(\omega,X)$-generic for $\tupof{L^X_n}{n<\omega}$.
   \end{proof}

  The following is also immediate by a very similar proof.
  
  \begin{proposition}\label{FA_coll_iff_DC_seq} If $\kappa$ is a regular cardinal, the following are equivalent
   $\FA({\mathit Coll_\kappa},{\mathfrak L}^\kappa)$, $\FA_{{\mathfrak L}^\kappa}({\mathit Coll_\kappa})$ and 
   $ \Ha_\kappa$.
    \end{proposition}
  
\begin{proof} 
  Suppose $\FA({\mathit Coll_\kappa},{\mathfrak L}^\kappa)$ holds and $X\not\prec \kappa$.
  Let $G$ be $\Coll(\kappa,X)$-generic for $\tupof{L^X_i}{i<\kappa}$.
  Then $\bigcup G$ is as required for the instance of $\Ha_\kappa$. 
  Moreover, by Lemma (\ref{injection_into_X_implies_density}), $\bigcup G$ shows that
  for each $i<\kappa$ we have $L^X_i$ is dense and hence
  $\FA_{{\mathfrak L}^\kappa}({\mathit Coll_\kappa})$ holds.

  It is trivial that $\FA_{{\mathfrak L}^\kappa}({\mathit Coll_\kappa})$ implies $\FA({\mathit Coll_\kappa},{\mathfrak L}^\kappa)$. 
  
  Finally suppose $\Ha_\kappa$ holds and $X\not\prec \kappa$. If $g:\kappa\longrightarrow X$ is
  as given by $\Ha_\kappa$, then, by Lemma (\ref{injection_into_X_implies_density}), for each
  $i<\kappa$ we have $L^X_i$ is dense in $\Coll(\kappa,X)$, and
  $\setof{g\on i}{i<\kappa}$ is  $\Coll(\kappa,X)$-generic for $\tupof{L^X_i}{i<\kappa}$. Thus
   $\FA_{{\mathfrak L}^\kappa}({\mathit Coll_\kappa})$ and $\FA({\mathit Coll_\kappa},{\mathfrak L}^\kappa)$ both hold.
    \end{proof} 

\begin{remark}\label{rmk_more_on_defn_coll} As noted in Remark (\ref{rmk_on_defn_coll}) it is important that the elements of
  the $\Coll(\kappa,X)$ are injections. Simply having ${^{<\kappa}X}$-generic for $\tupof{L^X_i}{i<\kappa}$ would not
  necessarily even allow us to use the generic to build a non-constant function, let alone an embedding of $\kappa$ into $X$.
  \end{remark}

\begin{corollary} $\AC \Longleftrightarrow$ there is a proper class of regular $\kappa$
  $\FA({\mathit Coll_\kappa},{\mathfrak L}^\kappa)$.
\end{corollary}

\begin{proof} Immediate from the previous proposition and Proposition (\ref{Hartogs_theorem}).
  \end{proof}

\vskip6pt

\begin{corollary} $\AC \Longleftrightarrow \hbox{  there is a proper class of regular } \kappa \sss\sss
  \FA_{\kappa}(\kappa\hbox{-closed})$
  \end{corollary}

\begin{proof} ``$\Longleftarrow$'' is immediate from the previous corollary. ``$\Longrightarrow$'' appears with a proof
  in \cite{Viale19}, Theorem (1.8), where it is seemingly attributed
  as an unpublished result due to Todor\v{c}evi\'c.\footnote{``The
  following formulation of the axiom of choice in terms of forcing
  axioms \emph{has been handed to me} by Todor\v{c}evi\'c, I’m not aware of any
  published reference.'' (\cite{Viale19}, p.7.). Our emphasis. \label{fn_Mateo_Stevo}}
  \end{proof}

\vskip12pt
\begin{definition} Let $\kappa$ be a regular cardinal and $X$ a set such that $X\not\prec \kappa$.
  Let 
  \[ \Qbb^X_\kappa = \setof{ t \in {^{<\kappa}([X]^{<\kappa})} }{
    \forall i\le \lh(t)\sss\sss \card{t(i) \setminus \bigcup_{j<i}t(j)} = 1 } ,\]
ordered by extension.

    Let $L^{X,\kappa,*}_i= \setof{t\in \Qbb^X_\kappa}{i\le \dom(t)}$ and
    $L^{X,\kappa,*} = \tupof{L^{X,\kappa,*}_i}{i<\kappa}$.
    Let $\Qbb_\kappa = \setof{\Qbb^X_\kappa}{X\not \prec \kappa}$. Let ${\mathfrak L}^{\kappa,*}$ be the Function
    such that if $X$ a set with $X\not\prec \kappa$ then ${\mathfrak L}^{\kappa,*}(X)=L^{X,\kappa,*}$.
  \end{definition}
\vskip6pt

\begin{lemma}\label{equiv_Coll_Qbb} Let $\kappa$ be a regular cardinal and $X$ a set such that $X\not\prec \kappa$.
  There is a definable isomorphism (of partial orders) between $\Coll(\kappa,X)$ and $\Qbb^X_\kappa$.
\end{lemma}

\begin{proof}     Given $f\in \Coll(\kappa,X)$ with $\dom(f)=\alpha$
  define $t^f = \tupof{\rge(f\on i+1)}{i\in \alpha}$ if $\alpha=\beta+1$
  If $t\in\Qbb^X_\kappa$ define $f^t \in {{^{<\kappa}}X}$ by
  $f^t(i)$ is the unique element of $t(i)\setminus \bigcup_{j<i} t(j)$ for $i<\lh(t)$.
\end{proof}

\vskip12pt

\begin{corollary}\label{coroll_1_to_equiv_Coll_Qbb} If $\kappa$ is a regular cardinal,
  $\FA(\Qbb_\kappa,{\mathfrak L^{\kappa,*}})\Longleftrightarrow
  \FA({\mathit Coll_\kappa},{\mathfrak L^\kappa})$.
  \end{corollary}

\begin{corollary}\label{coroll_2_to_equiv_Coll_Qbb} If $\kappa$ is a regular cardinal then
  $\Ha_\kappa \Longleftrightarrow \FA(\Qbb_\kappa,{\mathfrak L^{\kappa,*}})$.
  \end{corollary}

We next show how this result allows us to improve on Theorem (\ref{Diego325}).
\vskip6pt
  
  \begin{proposition}\label{prop_FA_Lambda_implies_FA_Coll} $\FA_\omega(\Lambda) \Longrightarrow \FA({\mathit Coll_\omega},{\mathfrak L}^\omega)$.
  \end{proposition}

  \begin{proof} Let $X$ be an infinite set. 
    Clearly $\Qbb^X_\omega\in \Lambda$, and so
    $\FA_\omega(\Lambda) \Longrightarrow \FA(\Qbb_\omega,{\mathfrak L}^{\omega,*})$.
    
    Consequently, by Lemma \ref{equiv_Coll_Qbb},
    $\FA_\omega(\Lambda) \Longrightarrow  \FA({\mathit Coll}_\omega,{\mathfrak L}^\omega)$.    
\end{proof}  
  \vskip6pt
  
  \begin{corollary} $\FA_\omega(\Lambda) \Longrightarrow \fin$.
    \end{corollary}

  \vskip6pt
    \begin{corollary} $\FA_\omega(\Lambda) \Longrightarrow \FA_\omega(\Gamma)$.
    \end{corollary}

    \vskip6pt
    
    Finally, for regular cardinals $\kappa$ we give a characterization of $\DC_\kappa$ in terms of a forcing axiom analogous to
    the characterization of $\Ha_\kappa$ by $\FA(\Coll_\kappa,{\mathfrak L}^\kappa)$ given in Proposition (\ref{FA_coll_iff_DC_seq}).
    We then use this to show $\FA_\omega(\Lambda) \Longleftrightarrow \DC$.
    \vskip6pt
    
    
    \vskip6pt
    \begin{definition}    Let $\kappa$ be a initial ordinal.

      For $X\not \prec \kappa$ let ${\mathcal F}^X_\kappa$ be the collection of
  functions  $F:(X)^{<\kappa}\longrightarrow {\mathcal P}(X)\setminus \set{\emptyset}$
  such that for all $t\in (X)^{<\kappa}$ and all $x\in F(t)$ we have $t\concat x \in (X)^{<\kappa}$.
  For $F\in {\mathcal F}^X_\kappa$ define $\Tbb(F) = \setof{t\in \Coll(\kappa,X)}{\forall i<\lh(t)\sss\sss t(i) \in F(t\on i)}$ with partial order inherited from
  $\Coll(\kappa,X)$.  Let ${\mathcal T}^X_\kappa = \setof{\Tbb(F)}{F\in {\mathcal F}^X_\kappa}$. 
 
   Let ${\mathcal T}_\kappa = \bigcup \setof{{\mathcal T}^X_\kappa }{X\not\prec \kappa}$.
    \end{definition}
    
    \vskip6pt
    \begin{definition} Let $\kappa$ be a initial ordinal.
      Let $\DC^*_{\kappa}$ be the assertion that for any $X$ with $X\not \prec \kappa$ and any function
      $F:(X)^{<\kappa}\longrightarrow {\mathcal P}(X)\setminus \set{\emptyset}$
      such that for all $t\in (X)^{<\kappa}$ and all $x\in F(t)$ we have $t\concat x \in (X)^{<\kappa}$
      there is some
      $g:\kappa\longrightarrow X$ such that for all $i<\kappa$ we have
      $g(i) \in F(g\on i)$.
    \end{definition}    

        \begin{remark} Let $\kappa$ be a cardinal. $\DC^*_{\kappa}$ is clearly equivalent to
      the assertion that for any $X$ with $X\not \prec \kappa$ and any function
      $F:{^{<\kappa}X}\longrightarrow {\mathcal P}(X)\setminus \set{\emptyset}$
      such that for all $t\in {^{<\kappa}X}$ we have $F(t)\subseteq F^X_{seq}(t)$
      there is some $g:\kappa\longrightarrow X$ such that for all $i<\kappa$ we have
      $g(i) \in F(g\on i)$.
        \end{remark}
        
    \vskip6pt
    \begin{proposition} Let $\kappa$ be a regular initial ordinal. Then $\DC_\kappa  \Longleftrightarrow
      \FA_{ {\mathfrak L}^\kappa}({\mathcal T}_\kappa)      \Longleftrightarrow
     \FA({\mathcal T}_\kappa, {\mathfrak L}^\kappa)
\Longleftrightarrow \DC^{*}_\kappa$.
    \end{proposition}

    \begin{proof} Suppose $\Tbb\in {\mathcal T}_\kappa$ and let $X\not\prec\kappa$ and $F\in {\mathcal F}^X_\kappa$ be such that
      $\Tbb=\Tbb(F)$.

      Applying $\DC_\kappa$ to $F$ there is some $g:\kappa \longrightarrow X$ such that for all $i<\kappa$ we have
      $g(i) \in F(g\on i)$. Then $G=\setof{g\on i}{i<\kappa}$ is $\Tbb$-generic for $\tupof{L^\kappa_i}{i<\kappa}$.

      Moreover, if $t\in \Tbb$ then consider the function $F_t\in {\mathcal F}^X_\kappa$ given by
      \begin{eqn}
        \begin{split}
          F_t(s) = & \sss\sss \set{t(i)}\hbox{ if }i<\lh(t)\hbox{ and }s =t\on i , \\
        = & \sss\sss F(s) \hbox{ otherwise.}
        \end{split}
      \end{eqn}

      Applying $\DC_\kappa$ to $F_t$ there is some $g_t:\kappa \longrightarrow X$ such that for all $i<\kappa$ we have
      $g_t(i) \in F(g_t\on i)$ and for $j\le\lh(t)$ we have $g_t\on j = t\on j$. Thus for each $\lh(t) \le i<\kappa$ we  
      $g_t\on i \le t$ and $g_t\on i \in L^X_i$. Hence for each $i<\kappa$ we have $L^X_i$ is dense in $\Tbb$.

      It is immediate that $ \FA_{ {\mathfrak L}^\kappa}({\mathcal T}_\kappa)      \Longleftrightarrow
     \FA({\mathcal T}_\kappa, {\mathfrak L}^\kappa)$
      
      Given $F$ as in the statement of $\DC^{*}_{\kappa}$  apply $\FA({\mathcal T}_\kappa, {\mathfrak L}^\kappa)$ to $\Tbb(F)$.
      The union of the generic filter will, as in the proof of Proposition (\ref{FA_coll_iff_DC_seq}), 
      give us an injection as required for $\DC^*_{\kappa}$.

      Finally, we need to show that $\DC^{*}_\kappa$ implies $\DC_\kappa$.

      The idea is to take sequences from ${^{<\kappa}X}$ and replace them by sequences from
      $(X\times \kappa)^{<\kappa}$, where the second coordinate contains a marker keeping score
      of how many times the element in the first coordinate has already appeared in the original sequence.
      (Notice that this is one of the few places where we do not have a ``set-by-set'' equivalence.)

      Suppose $F:{^{<\kappa}X}\longrightarrow {\mathcal P}(X)\setminus \set{\emptyset}$. 

      For $\alpha<\kappa$ and $t\in {^{\alpha}X}$ define $u^t \in (X\times \kappa)^{\alpha}$ by
      $u^t_i = \tup{t_i,\gamma^t_i}$ where $\gamma^t_i = \otp(\tupof{j<i}{t_j=t_i})$.
      
      We define $G:(X\times \kappa)^{<\kappa}\longrightarrow {\mathcal P}(X\times \kappa)\setminus \set{\emptyset}$
      by
      \begin{eqn}
        \begin{split}
          G(u) = & \sss\sss \setof{\tup{x,\delta^t_x}}{x\in F(t) \sss\sss\& \sss\sss
        \delta^t_x = \otp(\tupof{j<i}{t_j=x})}\hbox{ if }u = u^t_i , \\
        = & \sss\sss F^{X\times \kappa}_{seq} \hbox{ otherwise.}
        \end{split}
      \end{eqn}
      
      We can thus apply $\DC^*_\kappa$ to $G$ giving us a function $g:\kappa \longrightarrow X\times \kappa$
      such that for all $i<\kappa$ we have $g(i) \in G(g\on i)$.

      Now we can define $f:\kappa\longrightarrow X$ by $f(i) = g(i)_0$, the first coordinate of the pair $g(i)$.
      By induction we have that $f(i) \in F(f\on i)$ as required.
    \end{proof}

    \begin{remark} By a very similar proof we could also show for regular cardinals $\kappa$ the equivalence of $\DC_\kappa$
      and the forcing axiom stating that there is a $\Pbb$-generic for ${\mathfrak L}^\kappa$ for every $\Pbb$ which
      is a $\kappa$-closed sub partial order of $\Coll(\kappa,X)$ with no maximal branches for some $X\not\prec\kappa$.
    \end{remark}

    \vskip12pt
    
    \begin{corollary} $\DC \Longleftrightarrow \FA_\omega(\Lambda) $.
      \end{corollary} 

    \begin{proof} ``$\Longrightarrow$''. $\DC$ implies  $\FA_\omega(\setof{\Pbb }{\Pbb \hbox{ is a po} })$, and so
      implies $\FA_\omega(\Lambda)$.

      ``$\Longleftarrow$'' Suppose $\Tbb = \Tbb(F)\in {\mathcal T_\omega}$ for some
      $X\not\prec\kappa$ and $F\in  {\mathcal F}^X_\omega$. By definition we have
      $\Tbb = \setof{t\in \Coll(\omega,X)}{\forall i<\lh(t)\sss\sss t(i) \in F(t\on i)}$. We can apply the isomorphism of partial orders
      defined in Lemma (\ref{equiv_Coll_Qbb}) between $\Coll(\kappa,X)$ and $\Qbb^X_\kappa$ to $\Tbb$ to obtain an isomorphic subpartial order of 
      $\Qbb^X_\kappa$ which is a member of $\Lambda$. 
      
      Hence $\FA_\omega(\Lambda)$ implies $\FA({\mathcal T_\omega},{\mathfrak L^\omega})$, since for each
      $\Tbb\in {\mathcal T_\omega}$ and each $i<\omega$ the set $L^\omega_i$ is dense in $\Tbb$,
      and this in turn implies $\DC$ by the previous proposition in the case $\kappa=\omega$.
      \end{proof}

    \begin{corollary} $\AC \Longleftrightarrow \hbox{ for all regular }\kappa \hbox{ we have }
      \FA({\mathcal T}_\kappa,{\mathfrak L}^\kappa)$.
      \end{corollary} 
    

\vfill\eject

\Addresses


\begin{thebibliography}{99}

 \bibitem[Bomfim22]{Bomfim22} Diego Lima Bomfim,
   \emph{Axiomas de Forcing e Princípios de Escolha},
   Master's Dissertation,
   Universidade Federal de Bahia,
   (2022).

 \bibitem[Diener94]{Diener94} Karl-Heinz Diener,
   `A remark on ascending chain conditions, the countable axiom
   of choice and the principle of dependent choices',
   Mathematical Logic Quarterly,
   40(3),
   415-421,
   (1994).

 \bibitem[Jech73]{Jech73} Thomas Jech,
   \emph{The Axiom of Choice},
   North Holland,
   (1973).
   
 \bibitem[Hartogs1915]{Hartogs1915} Friedrich Hartogs,
   `Uber das Problem der Wohlordnung',
   Math. Annalen,
   76,
   436-443,
   (1915). 

 \bibitem[Herrlich06]{Herrlich06} Horst Herrlich,
   \emph{Axiom of choice},
   Lecture Notes in Mathematics,
   1876,
   Springer,
   (2006).

 \bibitem[Kunen80]{Kunen80} Ken Kunen,
   \emph{Set theory - an introduction to independence proofs},
   Elsevier,
   (1980).
   
 \bibitem[Levy64]{Levy64} Azriel L\'evy,
   ` The interdependence of certain consequences of the axiom of choice',
   Fundamenta mathematicae,
   54(2),
   135–157,
   (1964). 

   
 \bibitem[Parente15]{Parente15} Francesco Parente,
   \emph{Boolean valued models, saturation, forcing axioms},
   Ph.D. thesis,
   University of Pisa,
   (2015).

 \bibitem[Pincus69]{Pincus69} David Pincus,
   \emph{Individuals in Zermelo-Fraenkel set theory},
   Doctoral Dissertation,
   Harvard University,
   (1969).
   
 \bibitem[RubinRubin63]{RubinRubin63} Herman Rubin and Jean Rubin,
   \emph{Equivalents of the Axiom of Choice}.
   North Holland,
   (1963).

   
 \bibitem[Silva12]{Silva12} Samuel Gomes da Silva,
   `A topological statement and its relation to certain weak choice principles',
   Questions and Answers in General Topology,
   30(1),
   1-8,
   (2012).

 \bibitem[Shannon90]{Shannon90} Gary Shannon,
   `Provable forms of Martin's axiom',
   Notre Dame Journal of Formal Logic,
   31(3),
   382-388,
   (1990).
   
 \bibitem[Viale19]{Viale19} Matteo Viale,
   `Useful axioms',
   https://arxiv.org/abs/1610.02832v3,
   (2019).
      
\end{thebibliography}
\end{document}